\documentclass[11pt]{amsart}
\usepackage{amssymb,amsmath,epsfig}

\newtheorem{theorem}{Theorem}[section]
\newtheorem{prop}[theorem]{Proposition}
\newtheorem{lemma}[theorem]{Lemma}

\newcommand\beq{\begin{equation}}
\newcommand\eeq{\end{equation}}
\newcommand\bce{\begin{center}}
\newcommand\ece{\end{center}}
\newcommand\bea{\begin{eqnarray}}
\newcommand\eea{\end{eqnarray}}
\newcommand\bean{\begin{eqnarray*}}
\newcommand\eean{\end{eqnarray*}}
\newcommand\bmt{\begin{multline*}}
\newcommand\emt{\end{multline*}}
\newcommand\ben{\begin{enumerate}}
\newcommand\een{\end{enumerate}}
\newcommand\bit{\begin{itemize}}
\newcommand\eit{\end{itemize}}
\newcommand\brr{\begin{array}}
\newcommand\err{\end{array}}
\newcommand\bt{\begin{tabular}}
\newcommand\et{\end{tabular}}
\newcommand\nin{\noindent}

\newcommand\bs{\bigskip}
\newcommand\ms{\medskip}

\newcommand\ol{\overline}
\renewcommand\S{\mathcal S}

\newcommand\UD{\Lambda}
\newcommand\CUD{\Delta}
\newcommand\GCUD{\Theta}
\newcommand\fp{\operatorname{fp}}
\newcommand\exc{\operatorname{exc}}
\newcommand\ud{\operatorname{ud}}
\newcommand\nud{\operatorname{nud}}
\newcommand\lrm{\operatorname{lrm}}
\newcommand\extr{\operatorname{extr}}
\newcommand\st{\operatorname{mm}}
\newcommand\bij{\phi}
\newcommand\Jbij{\varphi}

\title{Cycle up-down permutations}
\author{Emeric Deutsch}
\author{Sergi Elizalde}

\begin{document}
\maketitle
\begin{abstract}
A permutation is defined to be {\em cycle-up-down} if it is a product of cycles that, when written starting with their smallest element,
have an up-down pattern. We prove bijectively and analytically that these permutations are enumerated by the Euler numbers,
and we study the distribution of some statistics on them, as well as on up-down permutations, on all permutations, and on a generalization
of cycle-up-down permutations.
The statistics include the number of cycles of even and odd length, the number of left-to-right minima, and the number of extreme elements.
\end{abstract}

\section{Introduction}\label{sec:intro}

Let $[n]=\{1,2,\dots,n\}$, and let $\S_n$ be the set of permutations of $[n]$.
A permutation $\pi$ of $[n]$ can be written in one-line notation as $\pi=\pi_1\pi_2\dots\pi_n$
or as a product of cycles as
$$\pi=(a_{11},a_{12},\dots)(a_{21},a_{22},\dots)(a_{31},a_{32},\dots)\dots$$
We will write commas in the cycle notation but not in the one-line notation, in order to distinguish them.
A cycle is said to be in {\em standard form} if its smallest element is in first position.
Every permutation $\pi$ has a unique expression as
a product of cycles in standard form where the first entries of the cycles are in increasing order.
For example, $\pi=2517364=(1,2,5,3)(4,7)(6)$. A {\em cyclic permutation} is a permutation that consists of only one cycle.

An {\em up-down} permutation of $[n]$ is a permutation $\pi$ satisfying $$\pi_1<\pi_2>\pi_3<\pi_4>\dots.$$
Let $\UD_n$ denote the set of up-down permutations of length $n$.
It is well known that the size of $\UD_n$ is the Euler number $E_n$.
The exponential generating function (EGF for short) for the Euler numbers is
 $$E(z)=\sec z+\tan z=\sum_{n\ge0} E_n\frac{z^n}{n!},$$
and the first values of $E_n$ are $1,1,1,2,5,16,61,272,\dots$. A permutation satisfying $\pi_1>\pi_2<\pi_3>\pi_4<\dots$ is called a {\em down-up} permutation.

In this paper we are concerned with a variation of the above definition. Instead of requiring the one-line notation to be up-down,
we will study permutations whose cycles,
when written in standard form, have an up-down pattern. A precise definition of these permutations, which we call {\em cycle-up-down permutations},
is given in Section~\ref{sec:CUD}. It is also shown, both analytically and bijectively, that they are enumerated by the Euler numbers.
In Section~\ref{sec:GCUD} we define a less restrictive family of permutations, which we call {\em generalized cycle-up-down permutations},
and enumerate them as well. In Section~\ref{sec:stat} we obtain refined generating functions for these families of permutations with respect to several statistics. We also find statistics that are preserved by our bijections between up-down and cycle-up-down permutations.
In Section~\ref{sec:statall} we study the distribution of some of these permutation statistics on $\S_n$, obtaining some formulas that involve the
Stirling numbers of the first kind. Finally, Section~\ref{sec:interpret} gives another interpretation of cycle-up-down permutations in terms
of perfect matchings.

\ms

Let us now introduce some notation and definitions that will be needed later on.
Let $A=\{a_1,a_2,\dots,a_n\}$ be a finite subset of the natural numbers, and assume that $a_1<\dots<a_n$.
Given a permutation $\sigma$ of $A$, let $\overline\sigma$ be the permutation obtained by replacing each entry $a_i$ of the one-line notation of $\sigma$ with $a_{n+1-i}$. For example, $\overline{2634}=6243$.
Following~\cite{J}, we will call this operation a {\em switch}.

We say that $\pi_i$ is a {\em left-to-right (LR) minimum} (resp. {\em maximum}) of a permutation $\pi$ if $\pi_i<\pi_j$ (resp. $\pi_i>\pi_j$) for all $1\le j<i$. If $i\ge2$, we say that $\pi_i$ is an {\em extreme element} if it is
either an LR minimum or an LR maximum (see~\cite[p.~98]{KPP}). The {\em min-max subsequence} of $\pi$ is defined to be the subsequence $\pi_{i_1}\pi_{i_2}\dots\pi_{i_k}$ where $\pi_{i_1}=1$ is the smallest entry of $\pi$,
$\pi_{i_2}$ is the largest entry to the right of $\pi_{i_1}$, $\pi_{i_3}$ is the smallest entry to the right of $\pi_{i_2}$, and so on. Note that we always have $i_k=n$. For example, the min-max sequence of $\pi=48127635$ is $1735$.
An {\em excedance} (resp. {\em deficiency}) of $\pi$ is a value $i$ such that $\pi_i>i$ (resp. $\pi_i<i$).
We use the term {\em odd cycle} (resp. {\em even cycle}) to mean a cycle of odd (resp. even) length.

For $\pi\in\S_n$, we define the following statistics, which will be studied in Sections~\ref{sec:stat} and~\ref{sec:statall}:
\bit
\item $c(\pi)=$ number of cycles of $\pi$,
\item $c_o(\pi)=$ number of odd cycles of $\pi$,
\item $c_e(\pi)=$ number of even cycles of $\pi$,
\item $\fp(\pi)=$ number of fixed points of $\pi$,
\item $\lrm(\pi)=$ number of left-to-right minima of $\pi$,
\item $\st(\pi)=$ length of the min-max sequence of $\pi$,
\item $\extr(\pi)=$ number of extreme elements of $\pi$,
\item $\exc(\pi)=$ number of excedances of $\pi$.
\eit

\section{Cycle-up-down permutations}\label{sec:CUD}

A cycle is said to be {\em up-down} if, when written in standard form, say $(b_1,b_2,\dots)$, one has
$b_1<b_2>b_3<\dots$.
We define a permutation $\pi$ to be {\em cycle-up-down} ({\em CUD} for short) if it is a product of up-down cycles.
For example, $(1,5,2,7)(3)(4,8,6)(9)$ is CUD, but $(1,3,5)(2,4)(6)$ is not. Let $\CUD_n$ be the set of CUD permutations of $[n]$.
The main result of this section is the enumeration of CUD permutations.

\begin{prop}\label{prop:CUD}
The number of CUD permutations of $[n]$ is $E_{n+1}$.
\end{prop}

In subsection~\ref{sec:CUDbij} we give a bijective proof of this fact. In order to construct a bijection,
it will be convenient to separate odd cycles from even cycles.
Up-down even cycles can alternatively be described as fully alternating cycles, meaning that in any representation of the cycle,
the entries alternate going up and down. Then, CUD permutations having only cycles of even length are precisely those permutations $\pi$
with no fixed points where the image of each excedance is a deficiency, and viceversa.
The following lemma enumerates these permutations.

\begin{lemma}\label{lemma:evenCUD}
The number of CUD permutations of $[2m]$ all of whose cycles are even is $E_{2m}$.
\end{lemma}

If we only allow odd cycles, we obtain a similar result.

\begin{lemma}\label{lemma:oddCUD}
 The number of CUD permutations of $[m]$ all of whose cycles are odd is $E_{m}$.
\end{lemma}

\subsection{Proofs via generating functions}

Let us first give relatively straightforward proofs of the above results using generating functions.

\begin{proof}[Analytic proof of Proposition~\ref{prop:CUD}]
The number of CUD cyclic permutations of length $n$ is $E_{n-1}$,
because each such permutation can be written as $(1,b_2,b_3,\dots,b_n)$,
with $(n+1-b_2)(n+1-b_3)\dots(n+1-b_n)$ being an up-down permutation of $[n-1]$.
Thus, the corresponding EGF is
\beq\label{eq:udcycles}\sum_{n\ge1} E_{n-1}\frac{z^n}{n!}=\int_0^z E(u)\,du.\eeq
Since any CUD permutation is an unordered product of up-down cycles, the EGF for CUD permutations is $$\exp\left(\int_0^z E(u)\,du\right),$$ via the set construction (see~\cite{FS}).
Thus, using that $\sum_{n\ge0} E_{n+1}\frac{z^n}{n!}=E'(z)$, the statement of the proposition, in terms of generating functions, is equivalent to
$$\exp\left(\int_0^z E(u)\,du\right)=E'(z).$$

Taking logarithms and differentiating, this equation is equivalent to $E(z)={E''(z)}/{E'(z)}$, which can be proved by simple calculus,
since $E'(z)=\sec z \tan z+\sec^2 z=E(z)\sec z$ and $E''(z)=E'(z)\sec z+E(z)\sec z\tan z=E(z)^2\sec z$.
\end{proof}

\begin{proof}[Analytic proof of Lemma~\ref{lemma:evenCUD}]
The EGF for CUD cyclic permutations of even length is
$$\sum_{m\ge1} E_{2m-1}\frac{z^{2m}}{(2m)!}=\int_0^z \tan u\,du.$$
Thus, the statement of the proposition is equivalent to  $$\exp\left(\int_0^z\tan u\,du\right)=\sec z,$$ which can be easily checked.
\end{proof}

\ms

\nin {\bf Remark.}
It is now clear that if we want to count permutations $\pi$
where the image of excedances are deficiencies and viceversa, we only need to modify the proof
of Lemma~\ref{lemma:evenCUD} by allowing fixed points. The EGF for these permutations is then
$$\exp\left(z+\int_0^z\tan u\,du\right)=e^z\sec z,$$
which appears in~\cite[A003701]{OEIS}.

\ms

\begin{proof}[Analytic proof of Lemma~\ref{lemma:oddCUD}]
The EGF for CUD cyclic permutations of odd length is
$$\sum_{m\ge0} E_{2m}\frac{z^{2m+1}}{(2m+1)!}=\int_0^z \sec u\,du.$$
Thus, the statement of the proposition is equivalent to $$\exp\left(\int_0^z\sec u\,du\right)=E(z),$$ which can be easily checked.
\end{proof}

\subsection{Bijective proofs}\label{sec:CUDbij}

Here we prove the three above results bijectively.

\begin{proof}[Bijective proof of Lemma~\ref{lemma:evenCUD}]
Given $\pi=\pi_1\pi_2\dots\pi_n\in\UD_{2m}$, let $\pi_{i_1}>\pi_{i_2}>\dots>\pi_{i_k}$ be its left to right minima. Note that all the $i_j$ must be odd. Then
\beq\label{eq:g}g(\pi)=(\pi_{i_1},\dots,\pi_{i_2-1})(\pi_{i_2},\dots,\pi_{i_3-1})\dots(\pi_{i_k},\dots,\pi_{2m})\eeq
is a CUD permutation with only even cycles, and $g$ is a bijection.

To find the preimage of a CUD permutation with only even cycles, write it as a product of cycles
in standard form ordered by decreasing first entry. Removing the
parentheses gives the one-line notation of an up-down permutation which is its preimage by $g$.

For example, if $\pi=47261538$, we have $g(\pi)=(4,7)(2,6)(1,5,3,8)$.
\end{proof}

\begin{proof}[Bijective proof of Lemma~\ref{lemma:oddCUD}]
It will be convenient to consider permutations of a finite set $A=\{a_1,a_2,\dots,a_n\}$ of natural numbers with $a_1<a_2<\dots<a_n$.
We describe a bijection $f$ between up-down permutations of $A$ and CUD permutations of $A$ with only odd cycles. Given an up-down permutation $\pi=\pi_1\pi_2\dots\pi_n$, let $k$ be such that $\pi_k=a_1$.
Define recursively $$f(\pi)=(\pi_k,\pi_{k-1},\dots,\pi_1)\,f(\overline{\pi_{k+1}\pi_{k+2}\dots \pi_n}),$$
where $\sigma\mapsto\overline\sigma$ is the switch operation defined in Section~\ref{sec:intro}.

For example, if $\pi=471938562$, then
\begin{multline*} f(471938562)=(1,7,4)f(\overline{938562})=(1,7,4)f(283659)=(1,7,4)(2)f(\overline{83659})\\
=(1,7,4)(2)f(59683)=(1,7,4)(2)(3,8,6,9,5),\end{multline*}
which is a CUD permutation. Conversely, given a CUD permutation of $A$ with odd cycles, write it as a product of cycles in standard form
ordered by increasing first entry, and reverse the previous steps. For example, if $\sigma=(1,8,5,7,2)(3,6,4)$,
then $$\sigma=(1,8,5,7,2)f(463)=(1,8,5,7,2)f(\overline{436})=f(27581436).$$
\end{proof}

The bijection $f$ is similar to Chebikin's decomposition of up-down permutations into a set of up-down permutations of odd length~\cite[p.~26]{Che}.

\begin{proof}[Bijective proof of Proposition~\ref{prop:CUD}]
We can combine the bijections in the above two proofs to give a bijection $\bij:\UD_{n+1}\rightarrow\CUD_n$.
In the above proof of Lemma~\ref{lemma:evenCUD} we described a bijection $g$
from up-down permutations of even length to CUD permutations with only even cycles.
Consider the straightforward generalization of this bijection to permutations of a finite set $A$ of natural numbers.

Given $\pi\in\UD_{n+1}$, let $k$ be such that $\pi_{k}=1$, and for $1\le i\le n+1$, let $\pi'_i=\pi_i-1$. Define
$\bij(\pi)$ to be the CUD permutation which has cycle decomposition given by
$$g(\pi'_1\dots\pi'_{k-1})f(\overline{\pi'_{k+1}\dots\pi'_{n+1}}).$$
For example, if $\pi=6\,9\,3\,8\,5\,12\,1\,10\,2\,11\,4\,7$, then
$$\bij(\pi)=g(5\,8\,2\,7\,4\,11)f(\overline{9\,1\,10\,3\,6}).$$
Now, noticing that the left-to-right minima of $5\,8\,2\,7\,4\,11$ are $5$ and $2$, we have
$$\bij(\pi)=(5,8)(2,7,4,11)f(3\,10\,1\,9\,6)=(5,8)(2,7,4,11)(1,10,3)(6)(9).$$
\end{proof}

An alternative bijective proof of Proposition~\ref{prop:CUD} can be obtained by modifying a bijection of Johnson~\cite{J} involving the so-called {\em zigzig sequences}.
We next describe the resulting map $\Jbij:\UD_{n+1}\rightarrow\CUD_n$.
Given $\pi=\pi_1\pi_2\dots\pi_{n+1}\in\UD_{n+1}$, let $\pi_k$ be its rightmost extreme element. If it is an LR minimum (i.e. $1$ in this case), let $\tau=\pi$;
if it is an LR maximum (i.e. $n$ in this case), let $\tau=\ol{\pi}$. We define $(\tau_k,\tau_{k+1},\dots,\tau_{n+1})$ to be the first cycle of $\Jbij(\pi)$, and we delete the entries
$\tau_k\tau_{k+1}\dots\tau_{n+1}$ from $\tau$. Now we look at the rightmost extreme element $\tau_j$ of what remains of $\tau$. Again, if it is an LR maximum, we switch $\tau$ to make $\tau_j$ an LR minimum.
The second cycle of $\Jbij(\pi)$ is $(\tau_j,\tau_{k+1},\dots,\tau_{k-1})$, we delete the entries $\tau_j\tau_{j+1}\dots\tau_{k-1}$ and repeat this process until $\tau$ has only one entry, which at this point is necessarily $n+1$.

For example, if $\pi=351827496$, then its rightmost extreme element is $9$, so we take $\tau=\ol{\pi}=759283614$, add the cycle $(1,4)$ to $\Jbij(\pi)$, and remove $1$ and $4$ from $\tau$.
Now the rightmost extreme element is $2$, so $\tau=7592836$, we add the cycle $(2,8,3,6)$ to $\Jbij(\pi)$, and delete $2836$ from $\tau$. The rightmost extreme element of $759$ is $9$,
so $\tau=\ol{759}=795$, and the cycle $(5)$ gets added to $\Jbij(\pi)$. The rightmost extreme element of $79$ is $9$,
so $\tau=\ol{79}=97$, the cycle $(7)$ gets added to $\Jbij(\pi)$, and we end here, obtaining $\Jbij(\pi)=(1,4)(2,8,3,6)(5)(7)$.

The inverse map can be described as follows. Given $\sigma\in\CUD_n$, write its cycles in standard form ordered by decreasing first element. Removing the parenthesis (this correspondence is due to Foata) gives
the one-line notation of a permutation. Let $\tau$ be the permutation obtained by adding $n+1$ in front of it. Let $k$ be the largest index such that $\tau_1\tau_2\dots\tau_k$ is alternating, that is, either up-down or down-up. Replace $\tau$ with
$\ol{\tau_1\tau_2\dots\tau_k}\tau_{k+1}\dots\tau_n$. Consider again the largest alternating prefix of $\tau$ and switch it. Repeat until $\tau$ is alternating. If it is up-down, then $\Jbij^{-1}(\sigma)=\tau$;
if it is down-up, then $\Jbij^{-1}(\sigma)=\ol{\tau}$.

Going back to the above example for $\sigma=(1,4)(2,8,3,6)(5)(7)$, after applying Foata's correspondence and inserting $9$ at the beginning we get $\tau=975283614$. It is alternating up to $\tau_2=7$, so we let
$\tau=\ol{97}5283614=795283614$. Now it is alternating up to $\tau_3=5$, so we let $\tau=\ol{795}283614=759283614$. This is already an alternating permutation. Since it is down-up, we recover $\Jbij^{-1}(\sigma)=\ol{\tau}=351827496$.

\section{Generalized cycle-up-down permutations}\label{sec:GCUD}

We say that a cycle is {\em generalized up-down} if it admits a representation $(a_1,a_2,\dots)$ satisfying $a_1<a_2>a_3<\dots$
(we call this an up-down representation of the cycle).
Note that up-down cycles are always generalized up-down cycles, but the converse is not true. For example, $(2,3,4,6)$ is a generalized up-down cycle,
but not an up-down cycle.
Define a permutation to be {\em generalized-cycle-up-down} (GCUD for short) if it is a product of generalized up-down cycles.
Let $\GCUD_n$ be the set of GCUD permutations of $[n]$.
In this section we enumerate GCUD permutations.

We start by enumerating GCUD cyclic permutations. First observe that generalized up-down cycles of odd length admit a unique up-down representation.
Indeed, given a cycle $(a_1,a_2,\dots,a_{2k+1})$ with $a_1<a_2>a_3<\dots>a_{2k+1}$, there is no
other representation of the cycle because either $a_{2k+1}<a_1<a_2$ or $a_{2k}>a_{2k+1}>a_1$ (unless $k=0$).
Thus, the number of GCUD cyclic permutations of odd length $2k+1$ is $E_{2k+1}$,
with EGF $\tan z$. The EGF for GCUD permutations having only odd cycles is $\exp(\tan z)$~\cite[A006229]{OEIS}.

To count GCUD cyclic permutations of even length we use the following result, which appears in \cite[Lemma~4]{Elkies}.
The sequence $kE_{2k-1}$ is \cite[A024255]{OEIS}.
\begin{lemma}\label{lem:uddecr}
The number of permutations $\pi\in\UD_{2k}$ with $\pi_{2k}>\pi_1$ is $k E_{2k-1}$.
\end{lemma}

\begin{proof}
Given an up-down permutation with $\pi_{2k}>\pi_1$, let $j$  be the index such that $\pi_j=1$ (note that $j$ must be odd). Then the permutation
$\pi_j \pi_{j+1}\dots \pi_{2k}\pi_1\dots \pi_{j-1}\in\S_{2k}$ is an up-down permutation beginning with $1$. Conversely, for any up-down permutation
$\sigma_1\sigma_2\dots\sigma_{2k}$ with $\sigma_1=1$, its $k$ cyclic rotations
$$\sigma_{2i-1}\sigma_{2i}\dots\sigma_{2k}\sigma_1\sigma_2\dots\sigma_{2i-2}$$
for $1\le i\le k$ are up-down permutations whose last entry is larger than the first.
Since the number of up-down permutations in $\S_{2k}$ beginning with $1$ is $E_{2k-1}$, the result follows.
\end{proof}

An immediate consequence of Lemma~\ref{lem:uddecr} is that the number of permutations $\pi\in\UD_{2k}$ with $\pi_{2k}<\pi_1$ is $E_{2k}-k E_{2k-1}$.

\begin{lemma}\label{lem:gudceven}
The number of GCUD cyclic permutations of $[2k]$ is $E_{2k}-(k-1)E_{2k-1}$.
\end{lemma}

\begin{proof}
Let $\gamma\in\S_{2k}$ be a GCUD cyclic permutation, and fix an up-down representation $(a_1,a_2,\dots,a_{2k})$ of $\gamma$.
If $a_{2k}<a_1$, then this is the only possible up-down
representation of $\gamma$, since any other representation would have the consecutive entries $a_{2k}<a_{1}<a_{2}$ (note that it
cannot start with $a_2$ because then the first and second entries would be $a_2>a_3$). By the above argument, there are $E_{2k}-k E_{2k-1}$
GCUD cyclic permutations of this kind.

If $a_{2k}>a_1$, then $\gamma$ has $k$ up-down representations. Picking the representation that starts with a $1$, there are $E_{2k-1}$ choices for
the other entries. In fact, these $E_{2k-1}$ GCUD cyclic permutations $\gamma$ are precisely the CUD cyclic permutations.
Altogether, the number of GCUD cyclic permutations of both kinds is
$E_{2k}-k E_{2k-1}+E_{2k-1}$.
\end{proof}

From Lemma~\ref{lem:gudceven} it is easy to compute the EGF for GCUD cyclic permutations of even length:
\begin{multline*}\sum_{k\ge1} (E_{2k}-(k-1) E_{2k-1})\frac{z^{2k}}{(2k)!}=
\sum_{k\ge1} E_{2k} \frac{z^{2k}}{(2k)!} - \sum_{k\ge1} \frac{1}{2} E_{2k-1}\frac{z^{2k}}{(2k-1)!}+\sum_{k\ge1} E_{2k-1} \frac{z^{2k}}{(2k)!}\\
=\sec z -1-\frac{z}{2}\tan z+\int_0^z\tan u\,du
=\sec z -1-\frac{z}{2}\tan z-\ln(\cos z).\end{multline*}
The beginning of the sequence, starting with the coefficient of $z$, is $0, 1, 0, 3, 0, 29, 0, 569, 0,\dots$. The term $3$, for example, corresponds to $(1,4,2,3)$, $(1,2,3,4)$, $(1,3,2,4)$.
The EGF for GCUD permutations having only even cycles is $$\sec(z)\,\exp\left(\sec z -1-\frac{z}{2}\tan z\right),$$
and the first few terms are $0, 1, 0, 6, 0, 89, 0, 2431, 0,\dots$. The term $6$ corresponds to
$(1,2)(3,4)$, $(1,3)(2,4)$, $(1,4)(2,3)$, $(1,2,3,4)$, $(1,3,2,4)$, $(1,4,2,3)$.

Finally, the EGF for all GCUD permutations is $$\sec(z)\,\exp\left(\sec z -1+(1-\frac{z}{2})\tan z\right).$$
The first few terms are $1, 2, 6, 21, 97, 491, 2989, 19756, 148444,\dots$. The term 21 corresponds to all permutations of $[4]$ except
$(1,2,4,3)$, $(1,3,4,2)$, $(1,4,3,2)$. None of the above three sequences appears currently in~\cite{OEIS}.

\section{Statistics on up-down, CUD, and GCUD permutations}\label{sec:stat}

\subsection{Refined generating functions}

It is possible to refine the above enumerations of CUD and GCUD permutations by adding a variable $t$ that marks the number of cycles
and a variable $x$ that marks the number of fixed points.

For CUD permutations we get the multivariate generating function
$$\sum_{n\ge0}\sum_{\pi\in\CUD_n} x^{\fp(\pi)}t^{c(\pi)}\frac{z^n}{n!}=\exp\left[t\left((x-1)z+\int_0^z E(u)\,du\right)\right]=e^{(x-1)tz}E'(z)^t=\frac{e^{(x-1)tz}}{(1-\sin z)^t}.$$
In particular, the EGF for CUD derangements is
$$\frac{e^{-z}}{1-\sin z},$$
and the first few terms are $0, 1, 1, 5, 15, 71, 341, 1945, 12135,\dots$.
The EGF for CUD permutations with respect to the number of cycles is
\beq\label{eq:cudcycles}(1-\sin z)^{-t}.\eeq
The coefficients of $z^n/n!$ in the expansion of (\ref{eq:cudcycles}) are the
polynomials $c_n$ in~\cite[Eq.~(7.1)]{J}.

It is also easy to keep track of odd and even cycles separately:
\beq\label{eq:cudoe}\sum_{n\ge0}\sum_{\pi\in\CUD_n} t_e^{c_e(\pi)}t_o^{c_o(\pi)}\frac{z^n}{n!}=\exp\left(\int_0^z (t_o\sec u+t_e\tan u)\,du\right)=(\sec z+\tan z)^{t_o}(\sec z)^{t_e}.\eeq

The number of odd cycles and the number of excedances in a permutation $\pi\in\CUD_n$ are related by $$c_o(\pi)+2\exc(\pi)=n,$$
since in each cycle of length $2k$ and each cycle of length $2k+1$, $k$ of the entries are excedances.
Consequently, the bivariate generating function for CUD permutations with respect to the number of excedances is
$$\sum_{n\ge0}\sum_{\pi\in\CUD_n} t^{\exc(\pi)}\frac{z^n}{n!}=\frac{[\sec(z\sqrt{t}) + \tan(z\sqrt{t})]^{1/\sqrt{t}}}{\cos(z\sqrt{t})}.$$

Finally, for GCUD permutations, the multivariate generating function keeping track of fixed points and cycles is
$$\sum_{n\ge0}\sum_{\pi\in\GCUD_n} x^{\fp(\pi)}t^{c(\pi)}\frac{z^n}{n!}=
\sec^t(z)\,\exp\left[t\left((x-1)z+\sec z -1+(1-\frac{z}{2})\tan z\right)\right].$$

\bs

\subsection{Statistics preserved by the bijections}

In the following statements, $f$, $g$, and $\bij$ are the bijections defined in Section~\ref{sec:CUDbij}.

\begin{lemma} Let $\pi\in\UD_{n}$ and let $f(\pi)\in\CUD_n$ be the corresponding CUD permutation with only odd cycles. Then
$$c(f(\pi))=\st(\pi).$$ 
\end{lemma}

\begin{proof}
We proceed by induction on $n$. The result is obviously true for $n=1$.
Recall the recursive definition of $f$: if $\pi\in\UD_n$ and $\pi_k=1$, then
$$f(\pi)=(\pi_k,\pi_{k-1},\dots,\pi_1)\,f(\overline{\pi_{k+1}\pi_{k+2}\dots \pi_n}).$$
Thus, we have that
$$c(f(\pi))=1+c(f(\overline{\pi_{k+1}\pi_{k+2}\dots \pi_n}))=1+\st(\overline{\pi_{k+1}\pi_{k+2}\dots \pi_n})=\st(\pi),$$
where the second equality holds by induction, and the last equality follows from the definition of the statistic $\st$.
\end{proof}

As a consequence of this lemma we obtain the EGF for up-down permutations with respect to the statistic $\st$:
\beq\label{eq:udmm}\sum_{n\ge0}\sum_{\pi\in\UD_n}t^{\st(\pi)}\frac{z^n}{n!}=\sum_{n\ge0}\sum_{\sigma\in\CUD_n}t^{c(\sigma)}\frac{z^n}{n!}=\exp\left(t\,\int_0^z\sec u\,du\right)=(\sec z+\tan z)^t,\eeq
where the second step follows from the analytic proof of Lemma~\ref{lemma:oddCUD}.
The coefficients of $z^n/n!$ in the expansion of (\ref{eq:udmm}) are the
polynomials $a_n$ in~\cite[Eq.~(2.1)]{J}.

\begin{lemma}
Let $\pi\in\UD_{2m}$ and let $g(\pi)\in\CUD_{2m}$ be the corresponding CUD permutation with only even cycles. Then
$$c(g(\pi))=\lrm(\pi).$$ 
\end{lemma}

\begin{proof}
It follows immediately from the definition of $g$, since the opening parentheses for the cycles of $g(\pi)$ are inserted at each left-to-right minimum of $\pi$.
\end{proof}

\begin{prop}\label{prop:statbij} Let $\pi\in\UD_{n+1}$ and let $\sigma=\bij(\pi)\in\CUD_n$. Then
\begin{eqnarray*} &&c_e(\sigma)=\lrm(\pi)-1,\\
&&c_o(\sigma)=\st(\pi)-1,\\
&&c(\sigma)=\lrm(\pi)+\st(\pi)-2.
\end{eqnarray*}
\end{prop}

\begin{proof}
Recall the definition of $\bij$: if $\pi\in\UD_{n+1}$ and $\pi_{k}=1$, then
$$\sigma=\bij(\pi)=g(\pi'_1\dots\pi'_{k-1})f(\overline{\pi'_{k+1}\dots\pi'_{n+1}}),$$
where $\pi'_i=\pi_i-1$ for all $i$.

Since the range of $f$ (resp. $g$) are CUD permutations with only odd (resp. even) cycles, we have that
$$c_e(\sigma)=c(g(\pi'_1\dots\pi'_{k-1}))=\lrm(\pi'_1\dots\pi'_{k-1})=\lrm(\pi_1\dots\pi_{k-1}1)-1=\lrm(\pi)-1$$
and
$$c_o(\sigma)=c(f(\overline{\pi'_{k+1}\dots\pi'_{n+1}}))=\st(\overline{\pi'_{k+1}\dots\pi'_{n+1}})=\st(1\pi_{k+1}\dots\pi_{n+1})-1=\st(\pi)-1.$$
The third equation is obtained by adding the other two.
\end{proof}

We know from equation~(\ref{eq:cudoe}) that the bivariate EGF for CUD permutations where $t$ marks the number of even cycles is
$$G(t,z)=(\sec z+\tan z)(\sec z)^t.$$
It follows from the first part of Proposition~\ref{prop:statbij} that the generating function for up-down permutations with respect to the number of LR minima is
$$\sum_{n\ge1}\sum_{\pi\in\UD_n}t^{\lrm(\pi)}\frac{z^n}{n!}=t\,\int_0^z G(t,u)\,du=t\,\int_0^z (\sec u+\tan u)(\sec u)^t\,du=t\,\int_0^z (\sec u)^{t+1}\,du+(\sec z)^t-1.$$
The coefficients of $z^{2n}/(2n)!$ in the expansion of the last expression are the polynomials $b_n$ in \cite[Eq.~(5.1)]{J}.

Under the bijection $\Jbij$ described at the end of Section~\ref{sec:CUDbij} and adapted from~\cite{J}, the number of cycles corresponds to a different statistic:
\begin{prop}\label{prop:statJbij}
Let $\pi\in\UD_{n+1}$ and let $\Jbij(\pi)\in\CUD_{n}$. Then
$$c(\Jbij(\pi))=\extr(\pi).$$
\end{prop}

\begin{proof}
The positions of the extreme elements of $\pi\in\UD_{n+1}$ are not affected by the switches that take place in the construction of $\Jbij(\pi)$.
Since a new cycle of $\Jbij(\pi)$ is created for each extreme element of $\pi$, the result follows.
\end{proof}

From the above proposition and equation~(\ref{eq:cudcycles}), it follows that the EGF for up-down permutations with respect to the number of extreme elements is
$$\sum_{n\ge1}\sum_{\pi\in\UD_n}t^{\extr(\pi)}\frac{z^n}{n!}=\int_0^z (1 - \sin u)^{-t}\,du.$$
It also follows from Propositions~\ref{prop:statbij} and~\ref{prop:statJbij} that the statistics $\extr$ and $\lrm+\st-2$ are equidistributed on up-down permutations.

\section{Statistics on all permutations}\label{sec:statall}

Some of the above statistics have an interesting distribution not only in $\UD_n$ but also in $\S_n$.
Let $c(n,k)$ denote the signless Stirling numbers of the first kind. It is well known that
$$|\{\pi\in\S_n:\lrm(\pi)=k\}|=|\{\pi\in\S_n:c(\pi)=k\}|=c(n,k).$$

\begin{lemma}\label{lem:stsn} For $n\ge k\ge1$,
$$|\{\pi\in\S_n:\st(\pi)=k\}|=c(n,k).$$
\end{lemma}

\begin{proof}
We define a bijection $h:\S_n\rightarrow\S_n$ that transforms the min-max sequence of a permutation $\pi$ into the sequence of left-to-right minima of $h(\pi)$.
Given $\pi\in\S_n$, let $\pi_{i_1}\pi_{i_2}\dots\pi_{i_k}$ be the min-max sequence of $\pi$.
Starting with $\tau=\pi$, for $j=1,2,\dots,k-1$ switch the entries $\tau_{i_j+1}\tau_{i_j+2}\dots\tau_n$ of $\tau$.
Finally, if $\tau_1\tau_2\dots\tau_n$ is the resulting permutation, let $h(\pi)=\tau_n\dots\tau_2\tau_1$.

For example, if $\pi=48{\bf 1}2{\bf 7}6{\bf 35}$ (the min-max sequence is in boldface), we first switch
the entries $27635$, getting $\tau=48{\bf1}7{\bf2}3{\bf65}$, then switch the entries $365$, getting $\tau=48{\bf1}7{\bf2}6{\bf35}$. The switch of the $5$ at the end does not change the permutation,
so we obtain $h(\pi)={\bf 53}6{\bf 2}7{\bf 1}84$, where now the elements in boldface are the LR minima.
\end{proof}

In fact, the following generalization of Lemma~\ref{lem:stsn} holds as well. For any infinite sequence $s=s_1s_2s_3\dots$, where each $s_i\in\{\min,\max\}$, define the
statistic $m_s$ on permutations $\pi\in\S_n$ as the length of the subsequence $\pi_{i_1}\pi_{i_2}\dots\pi_{i_k}$ where $\pi_{i_1}=s_1(\pi_1\pi_2\dots\pi_n)$,
$\pi_{i_2}=s_2(\pi_{i_1+1}\pi_{i_1+2}\dots\pi_n)$, and in general
$\pi_{i_j}=s_j(\pi_{i_{j-1}+1}\pi_{i_{j-1}+2}\dots\pi_n)$ for each $j$ until $\pi_{i_k}=\pi_n$ for some $k$. For example, if $s=\min\min\dots$, then $m_s(\pi)=\lrm(\pi)$, and if
$s=\min\max\min\max\dots$, then $m_s(\pi)=\st(\pi)$. Then, for any sequence $s$ as above, we have
$$|\{\pi\in\S_n:m_s(\pi)=k\}|=c(n,k).$$
The bijection from Lemma~\ref{lem:stsn} can be easily generalized to prove this fact if we start with $\tau=\pi$ or $\tau=\ol{\pi}$ depending on whether $s_1$ equals $\min$ or $\max$, respectively, and
then for each $j=1,2,\dots,k-1$ we switch the entries $\tau_{i_j+1}\tau_{i_j+2}\dots\tau_n$ only if $s_j\neq s_{j+1}$.

\begin{prop} For $n,k\ge1$ with $n\ge k+1$, $$|\{\pi\in\S_n:\extr(\pi)=k\}|=2^k c(n-1,k).$$
\end{prop}

\begin{proof}
We give a bijection $\ell$ between the sets
$$\{(\pi,s): \pi\in\S_{n-1}, \lrm(\pi)=k, s\in\{0,1\}^k\}$$
and $$\{\sigma\in\S_n: \extr(\sigma)=k\}.$$

Given a pair $(\pi,s)$ in the first set, the digit $s_j$ being $0$ or $1$ will indicate whether the $j$th LR minimum of $\pi$ will become an LR minimum or an LR maximum of $\ell(\pi,s)$, respectively.
Let $\pi_{i_1}>\pi_{i_2}>\dots>\pi_{i_k}$ be the LR minima of $\pi$.
Define $s_{k+1}=0$ for convenience.
Let $\tau=\tau_0\tau_1\dots\tau_{n-1}=n\pi_1\dots\pi_{n-1}$.
For $j=k,k-1,\dots,1$, switch the prefix $\tau_0\tau_1\dots\tau_{i_j}$ of $\tau$ if $s_j\neq s_{j+1}$.
Let $\sigma=\ell(\pi,s)$ be the permutation $\tau$ obtained at the end.
It is clear from the construction that the positions of the extreme elements of $\sigma$ are the same as the LR minima of $\pi$.

For example, if $(\pi,s)=({\bf 86}7{\bf 42}5{\bf 1}3,10011)$ (the LR minima of $\pi$ are in boldface), we start with
$\tau=9{\bf86}7{\bf42}5{\bf1}3$.
Since $s_5=1\neq0=s_6$, we first switch the prefix $9867451$, getting $\tau=1{\bf25}4{\bf78}6{\bf9}3$. Now $s_4=1=s_5$, so no switches are done for $j=4$.
Since $s_3=0\neq1=s_4$, we now switch the prefix $12547$, getting $\tau=7{\bf52}4{\bf18}6{\bf9}3$. The last switch happens for $j=1$ because $s_1\neq s_2$,
ending with $\sigma=\ell(\pi,s)=\tau=5{\bf72}4{\bf18}6{\bf9}3$.

The inverse map has a similar description. Given $\sigma\in\S_n$ with $k$ extreme elements, we have $\ell^{-1}(\sigma)=(\pi,s)$ where, for $1\le j\le k$,
$s_j$ is $0$ or $1$ depending on whether the $j$th extreme element of $\sigma$ is an LR minimum or an LR maximum, respectively.
To obtain $\pi$, let $\sigma_{i_1},\sigma_{i_2},\dots,\sigma_{i_k}$ be the extreme elements of $\sigma$ from left to right.
For $j=k,k-1,\dots,1$, switch the first $i_j$ entries of $\sigma$ if $s_j\neq s_{j+1}$, where again we define $s_{k+1}=0$ for convenience.
Finally, we recover $\pi$ by deleting the first entry, which is necessarily $n$.
\end{proof}

Let us now study the distribution of two more statistics on permutations.
For $\pi\in\S_n$, let $\ud(\pi)$ be the number of up-down cycles in $\pi$, and let
$\nud(\pi)=c(\pi)-\ud(\pi)$ be the number of cycles of $\pi$ that are not up-down. Using equation~(\ref{eq:udcycles}) and the fact that
the EGF for all cycles is $$\sum_{n\ge1}(n-1)!\,\frac{z^n}{n!}=-\ln(1-z),$$
we can derive the EGF for all permutations with respect to the number of up-down and non-up-down cycles:
\begin{multline}\label{eq:udnud}
\sum_{n\ge0}\sum_{\pi\in\S_n} v^{\ud(\pi)}w^{\nud(\pi)}\frac{z^n}{n!}=
\exp\left(-w\ln(1-z)+(v-w)\int_0^z E(u)\,du\right)\\=
\frac{E'(z)^{v-w}}{(1-z)^w}=\frac{1}{(1-z)^w(1-\sin z)^{v-w}}.\end{multline}

It is now straightforward to obtain an expression for the expected number of up-down cycles in a random permutation.
Plugging $w=1$ into~(\ref{eq:udnud}) we get $$\frac{(1-\sin z)^{1-v}}{1-z},$$ the EGF for permutations where $v$ marks
the number of up-down cycles. Its derivative with respect to $v$, evaluated at $v=1$, is
\beq\label{eq:avgud}-\frac{\ln(1-\sin z)}{1-z}=\left(\sum_{i\ge1} E_{i-1}\frac{z^i}{i!}\right)\left(\sum_{j\ge0} z^j\right).\eeq
The expected number of up-down cycles in a random permutation of $[n]$ is now the coefficient of $z^n$ in~(\ref{eq:avgud}), which equals
\beq\label{eq:avgudcoef}\frac{E_0}{1!}+\frac{E_1}{2!}+\dots+\frac{E_{n-1}}{n!}.\eeq
This formula can also be obtained directly using the well-known fact that the expected number of cycles of length $k$ in a random permutation of
$[n]$ is $1/k$, for all $1\le k\le n$. For each one of these cycles, the probability that it is up-down is
$E_{k-1}/(k-1)!$, since there are $E_{k-1}$ up-down cycles of length $k$. Thus,
the expected number of up-down cycles of length $k$ in a random permutation of $[n]$ is $E_{k-1}/k!$.
Adding up for $1\le k\le n$ we obtain~(\ref{eq:avgudcoef}).  For large $n$, this number approaches
$$-\ln(1 - \sin 1) = 1.841817641\dots.$$

Another consequence of equation~(\ref{eq:udnud}) is a formula for the number $r_n$ of permutations of $[n]$ having no up-down cycles.
The EGF for these permutations, obtained by setting $v=0$ and $w=1$ in~(\ref{eq:udnud}), is
$$\frac{1 - \sin z}{1-z}=\left(1+\sum_{i\ge1} (-1)^i\frac{z^{2i-1}}{(2i-1)!}\right)\left(\sum_{j\ge0} z^j\right).$$
It follows that
$$\frac{r_{2m-1}}{(2m-1)!}=\frac{r_{2m}}{(2m)!}=\frac{1}{3!}-\frac{1}{5!}+\frac{1}{7!}-\dots+\frac{(-1)^m}{(2m-1)!}.$$
The limiting value of this expression, which is the probability that a random permutation of length approaching infinity has no up-down cycles,
is $1 - \sin 1 = 0.1585290152\dots$.

\section{Another interpretation of CUD permutations}\label{sec:interpret}

A permutation of $[n]$ can be represented as a directed graph on $n$ vertices labeled $1,2,\dots,n$, with an edge from $i$ to $j$ if and only if $\pi(i)=j$.
Consider the following way of drawing the graph. Put the $n$ vertices on a horizontal line, ordered from left to right by increasing label. If $\pi(i)=j>i$ (resp. $\pi(i)=j<i$),
then draw a red (resp. blue) arc between $i$ to $j$ above (resp. below) the horizontal line.
Note that the orientation of the arcs is implicit, so we can draw undirected arcs (see Figure~\ref{fig:CUDperm}).

\begin{figure}[hbt]
\bce\epsfig{file=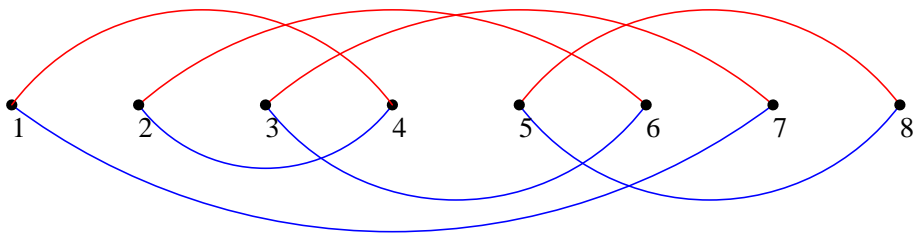,height=3cm} \caption{The CUD permutation $(1,4,2,6,3,7)(5,8)$. }\label{fig:CUDperm}\ece
\end{figure}

Using this drawing, CUD permutations having only even cycles correspond precisely to those graphs with the property that each vertex is of one of two types:
it is incident to a red and a blue edge that connect it to vertices to its left, or it is incident to a red and a blue edge that connect it to vertices to its right.
Equivalently, the edges of each color form a perfect matching of the vertices, and the two matchings agree on what vertices are {\em opening vertices} (matched with a vertex to their right)
or {\em closing vertices} (matched with a vertex to their left).

Since the number of CUD permutations of $[2m]$ having only even cycles is $E_{2m}$, as shown in Lemma~\ref{lemma:evenCUD}, this representation as a pair of ``matching" perfect matchings
allows us to recover the continued fraction for the secant numbers (see~\cite[p.~145]{Fla}):
$$\sum_{m\ge0} E_{2m} z^m=\dfrac{1}{1-\dfrac{z}{1-\dfrac{2^2z}{1-\dfrac{3^2z}{1-\dfrac{4^2z}{\dots}}}}}.$$

\end{document}